\documentclass[11pt]{amsart}

\usepackage{amscd}
\usepackage{amsmath, amssymb,color,stmaryrd}
\usepackage{amsfonts, mathrsfs}
\usepackage{cancel}
\usepackage{comment}

\usepackage[pdftex,hyperindex]{hyperref}
\usepackage[backend=biber,maxnames=5,minnames=5,minbibnames=5,maxbibnames=5,maxalphanames=5,giveninits=true,doi=false,url=false,isbn=false,style=alphabetic,date=year,sorting=nyt]{biblatex}
\renewbibmacro{in:}{} 

\usepackage{geometry}
\newcommand{\dbar}{\overline{\partial}}

\newcommand{\p}{\partial}
\newcommand{\pbar}{\ov{\partial}}

\newcommand{\abs}[1]{\left\lvert#1\right\rvert}
\newcommand{\norm}[1]{\left\lVert#1\right\rVert}

\newcommand{\ov}[1]{\overline{#1}}
\newcommand{\ul}[1]{\underline{#1}}

\newcommand{\tr}[2]{\textrm{tr}_{#1}{#2}}
\newcommand{\ti}[1]{\widetilde{#1}}

\newcommand{\e}{\varepsilon}

\newcommand{\R}{\mathbb{R}}
\newcommand{\C}{\mathbb{C}}

\renewcommand{\leq}{\leqslant}
\renewcommand{\geq}{\geqslant}
\renewcommand{\le}{\leqslant}
\renewcommand{\ge}{\geqslant}

\newcommand{\be}{\begin{equation}}
\newcommand{\ee}{\end{equation}}

\newcommand{\Hom}{\mathrm{Hom}}

\newcommand{\wk}{\rightharpoonup}{}

\begin{document}
\newtheorem*{thmA}{Theorem A} 
\newtheorem*{thmAp}{Theorem A'} 
\newtheorem*{thmB}{Theorem B} 
\newtheorem*{thmBp}{Theorem B'} 
\newtheorem*{thmC}{Theorem C} 
\newtheorem{claim}{Claim}
\newtheorem{theorem}{Theorem}[section]
\newtheorem{lemma}[theorem]{Lemma}
\newtheorem{corollary}[theorem]{Corollary}
\newtheorem{proposition}[theorem]{Proposition}
\newtheorem{question}{question}[section]
\theoremstyle{definition}
\newtheorem{definition}[theorem]{Definition}
\newtheorem{remark}[theorem]{Remark}

\numberwithin{equation}{section}

\title[Geodesic rays]{Geodesic Rays in the Donaldson--Uhlenbeck--Yau Theorem}

\begin{abstract}
We give new proofs of two implications in the Donaldson--Uhlenbeck--Yau theorem. Our proofs are based on geodesic rays of Hermitian metrics, inspired by recent work on the Yau--Tian--Donaldson conjecture.
\end{abstract}

\author{Mattias Jonsson \and Nicholas McCleerey \and Sanal Shivaprasad}

\address{Dept of Mathematics\\
  University of Michigan\\
  Ann Arbor, MI 48109-1043}
\email{mattiasj@umich.edu,njmc@umich.edu,sanal@umich.edu}

\subjclass[2020]{Primary: 53C07, Secondary: 14J60, 32Q15, 58E15}

\maketitle

\setcounter{tocdepth}{1}


\section{Introduction}


 Let $(X^n, \omega)$ be a compact K\"ahler manifold,
 and $E$ a holomorphic vector bundle of rank $r$ over $X$. A \emph{Hermite-Einstein metric} on $E$ is a Hermitian metric $h$ which satisfies
  \[
   \Theta(h) \wedge \omega^{n-1} = \gamma\, \omega^{n} \mathrm{Id}_{E},
 \]
 where $\Theta(h)$ is the curvature of $h$, $\mathrm{Id}_E$ is the identity endomorphism, and $\gamma$ is a cohomological constant.

The celebrated Donaldson--Uhlenbeck--Yau theorem states that $E$ admits a Hermite-Einstein metric if and only if $E$ is slope stable \cite{Don87,UY86}. We will consider the following version:
\begin{theorem}\label{DUY}
 
Suppose that $E$ is a
holomorphic vector bundle over a compact K\"ahler manifold $(X^n, \omega)$. Then the following conditions are equivalent:
\begin{enumerate}
\item $E$ is slope stable;
\item The Donaldson functional $\mathcal{M}$ is proper on the space of Hermitian metrics; 
\item $E$ admits a unique Hermite-Einstein metric.
\end{enumerate}
\end{theorem}

Here $E$ is \emph{slope stable} if and only if $\mu_{\mathcal{E}}<\mu_E$ for any nontrivial holomorphic torsion free subsheaf $\mathcal{E}\subset E$, where $\mu_{\mathcal{F}}$ denotes the slope of a holomorphic sheaf $\mathcal{F}$ over $X$ with respect to $\omega$. The \emph{Donaldson functional} $\mathcal{M}$ is a functional on the space of Hermitian metrics whose minimizers are exactly the Hermite--Einstein metrics. See~\S\ref{sec:background} for details.

The equivalence (1)$\Leftrightarrow$(3) in Theorem~\ref{DUY}
was the first result to link solvability of a geometric PDE to a stability condition in the sense of GIT, and has proven to be deeply influential in shaping subsequent results and conjectures, e.g. the Yau--Tian--Donaldson conjecture, as discussed shortly.

We have chosen the above formulation of Theorem~\ref{DUY} in order to emphasize the similarities with the variational approach to the Yau--Tian--Donaldson (YTD) conjecture on the existence of (unique) cscK metrics on polarized complex manifolds $(X,L)$. Despite much recent progress, this conjecture is still open in general, but it is settled for Fano manifolds, when $L=-K_X$, see~\cite{Ber16,CDS15,CDS15a,CDS15b,Tia15,DS16,CSW18,BBJ21}, and even for (possibly singular) log Fano pairs~\cite{Li22a}.

In the general YTD conjecture, the Donaldson functional is replaced by the Mabuchi K-energy functional, and slope stability by a suitable version of K-stability, a condition on the space of test configurations for $(X,L)$~\cite{Tia97,Don02}.

The analogue of the equivalence~(2)$\Leftrightarrow$(3) is known in full generality~\cite{BBE+19,DR17,CC21,CC21a}, and versions of~(3)$\Rightarrow$(1) and (1)$\Rightarrow$(2) have been shown in~\cite{BHJ19} and~\cite{Li21}, respectively (albeit with two different, conjecturally equivalent \cite{Li21,BJ22}, definitions of stability). In each of these proofs, the notion of a \emph{geodesic ray} in the space of singular semipositive metrics on $L$ plays a crucial role.

Going back to the Hermite--Einstein problem, there is also a natural notion of geodesic rays in the space of Hermitian metrics on $E$, and the goal of the present paper is to give new proofs of the implications~(3)$\Rightarrow$(1) and (1)$\Rightarrow$(2) by utilizing geodesic rays, paralleling the recent work on the YTD conjecture.

Our first result constructs a geodesic ray from any filtration. As in the YTD conjecture, we will actually construct rays of singular metrics. Let $\mathcal{H}^{1,p}$ be the space of trace-free endomorphisms on $E$ with coefficients in the Sobolev space $W^{1,p}$, $1 \leq p \leq \infty$. 

\begin{thmA}
Let $h_0$ be a hermitian metric on $E$ and \[0 =: \mathcal{E}_{m+1} \subset \mathcal{E}_m\subset \ldots \subset \mathcal{E}_{1} := E\] a filtration of $E$ by holomorphic subsheaves. Let $\mathcal{F}_i := \mathcal{E}_{i}/\mathcal{E}_{i+1}$, $1\leq i \leq m$. Then there exists $w\in\mathcal{H}^{1,\infty}$, such that the geodesic ray of singular hermitian metrics $h_t:= e^{tw}h_0$, $t\geq 0,$ satisfies:
\begin{equation}\label{slopes}
\lim_{t\rightarrow \infty}\frac{\mathcal{M}_{\omega}(h_t, h_0)}{t} = \sum_{k=1}^m 2\pi (m-k+1)\, \mathrm{rk}(\mathcal{F}_k)\, (\mu_{\mathcal{F}_k} - \mu_{E}),
\end{equation}
where $\mathrm{rk}(\mathcal{F}_k)$ is the rank of $\mathcal{F}_k$, and $\mu_{\mathcal{F}_k}$ is its slope (with respect to $\omega$).
\end{thmA}

Theorem~A, which can be viewed as an analogue of Theorem~A in~\cite{BHJ19}, easily gives the implication $(3)\Rightarrow (1)$ in Theorem~\ref{DUY}, using the fact that Hermite--Einstein metric are exactly the minimizers of the Donaldson functional, which is furthermore convex along any geodesic ray. Indeed, when $m=2$, the sign of the right-hand side of~\eqref{slopes} is evidently related to the slope stability of $E$, since for any $\mathcal{E}\subset E$, $\mathrm{rk}(\mathcal{E})(\mu_E - \mu_{\mathcal{E}}) = \mathrm{rk}(E/\mathcal{E})(\mu_{E/\mathcal{E}} - \mu_E)$.

We construct the desired ray in Theorem~A by first passing to a smooth resolution $\pi:\ti{X}\rightarrow X$ of the filtration (see e.g. \cite{Jac14, Sib15}). We then explicitly describe a smooth endomorphism $w$ on $\pi^*E$ which acts on $\pi^* h_0$ by scaling the induced metrics on the quotient bundles $\pi^*\mathcal{F}_{k}$ by $k$; pushing $w$ forward produces the desired ray downstairs.

\medskip
As we see from Theorem~A, it is useful to consider geodesics of singular metrics of the form $e^{tw}h_0$, for some $w\in\mathcal{H}^{1,p}$ -- note that $e^{tw}h_0$ will generally be much more singular than $w$. With this setup, the largest natural space of metrics to consider is:
\[
\mathcal{S} := \{w\in \mathcal{H}^{1,1}\ |\ \mathcal{M}(e^wh_0, h_0) < \infty\}.
\]
Proposition \ref{propReverseSobolev} (\cite{Don87}) shows that actually $\mathcal{S}\subset \mathcal{H}^{1,p_{\max}}$, where $p_{\max}:= \frac{2n}{2n-1}$, and it seems natural to interpret $\mathcal{H}^{1,p}$ as an analogue of the space $\mathcal{E}^1$ of metrics (or potentials) of finite energy in the study of cscK metrics (although it is not obvious to us that the constant $p_{\max}$ is optimal, c.f. Theorem~A). Our next result can be viewed as an analogue of~\cite[Theorem~2.16]{BBJ21}.

\begin{thmB}\label{thmB}
If the Donaldson functional $\mathcal{M}$ is non-proper on $\mathcal{S}$ with respect to the $W^{1,p}$-norm, for any $1 < p < p_{\max}$, then there exists a geodesic ray in $\mathcal{S}$ along which $\mathcal{M}$ is bounded above.
\end{thmB}

Note that Proposition \ref{propReverseSobolev} actually implies that $\mathcal{M}$ will be proper with respect to the $W^{1,p_{\max}}$-norm if $E$ admits an HE metric, so that Theorem~B can likely be improved.

There are two main ingredients in the proof of Theorem~B. The first is the lower semicontinuity of $\mathcal{M}$ in the weak $W^{1,p}$-topology, a fact for which we give a new, elementary proof, see Proposition~\ref{prop:Mlsc}. The analogous fact in the cscK case is that the Mabuchi functional is lsc with respect to the strong topology on the space of metrics of finite energy, see~\cite{BBE+19}. The second is a compactness statement, which here boils down to the Banach--Alaoglu theorem. 
The analogue in the cscK case is that sets of bounded entropy are strongly compact, as proved in~\cite{BBE+19}.

\medskip
Finally we go from geodesic rays to filtrations:
\begin{thmC}
Suppose that $E$ admits a geodesic ray in $\mathcal{S}$ along which the Donaldson functional is bounded from above. Then there exists a nontrivial filtration of $E$ by holomorphic subsheaves $\{\mathcal{E}_k\}_{k=1}^{m+1}$ such that $\mu_{\mathcal{E}_k}\ge\mu_E$ for at least one $k$. In particular, $E$ is not slope stable. 
\end{thmC}

Theorem~C follows from a formula for $\mathcal{M}(h_t,h_0)$ in terms of the eigenvalues of $\log(h_0h_t^{-1})$, due to Donaldson \cite{Don87}. Using this, we show that $\mathcal{M}(h_t)$ can only be bounded from above under very restrictive circumstances: essentially, the geodesic ray $(h_t)$ must have come from a construction similar to Theorem~A. The weakly holomorphic $W^{1,2}$-projection theorem of Uhlenbeck-Yau \cite{UY86, UY89} can then be used to produce the desired filtration; applying \eqref{slopes} shows that at least one of the subsheafs in the filtration has slope larger than $\mu_E$. 

The role of Theorem~C in the cscK case $(X,L)$ is played by Theorem~6.4 in~\cite{BBJ21}, which to any geodesic ray (of linear growth) of metrics of finite energy associates a psh metric of finite energy on the Berkovich analytification of the line bundle with respect to the (non-Archimedean) trivial absolute value on $\mathbf{C}$. As proved by C.~Li in~\cite{Li21}, the slope at infinity of the Mabuchi functional along the ray is bounded below by the Mabuchi functional evaluated at the non-Archimedean metric.
In the setting of Theorem~C, the limiting object is simpler, given by a filtration of $E$ by holomorphic subsheaves.

The combination of Theorems~B and~C evidently give us the implication (1)$\Rightarrow$(2) in Theorem~\ref{DUY}. As already mentioned, (3)$\Rightarrow$(1) follows from Theorem~A. The remaining implication (2)$\Rightarrow$(3) can be shown by an easy application of the Hermitian-Yang-Mills flow~\cite{Don85}. In the K\"ahler-Einstein case, any minimizer of the Mabuchi functional is a K\"ahler--Einstein metric~\cite{BBGZ13}, and the corresponding result in the cscK case holds as well~\cite{CC21,CC21a}. It is reasonable to believe that a minimizer in $\mathcal{H}^{1,p}$ of the Donaldson functional must in fact be a (smooth) Hermitian metric.

\medskip
\noindent {\bf Comparison with previous works:} In terms of history, L\"ubke \cite{Lub83} and Kobayashi \cite{Kob87} first proved the implication (3)$\Rightarrow$(1) in Theorem \ref{DUY} by using vanishing theorems for $E$. The more difficult implication (1)$\Rightarrow$(3) was proved by Donaldson~\cite{Don83} for projective surfaces and by Uhlenbeck and Yau~\cite{UY86} in general. 
Donaldson gave another proof in~\cite{Don87} for projective manifolds, using induction on dimension and a theorem of Mehta--Ramanathan \cite{MR84}, and the implications (1)$\Rightarrow$(2) and (2)$\Rightarrow$(3) can be extracted from results in that paper. 

Subsequent work of Simpson simultaneously unified and generalized the approaches in \cite{UY86} and \cite{Don87}, establishing a version of Theorem \ref{DUY} for Higgs bundles over certain non-compact K\"ahler manifolds. His usage of a blow-up argument along a non-proper ray to extract a limiting endomorphism and subsequent filtration is similar in spirit to our proof of Theorem $B$, but with several differences; firstly, his notion of properness is less general than ours, only holding for a special subclass of metrics with $L^1$-curvature. Several simplifications follow from this -- for instance, it is easy to show the lower semi-continuity of $\mathcal{M}$ on this smaller space, and he has no need to work with regularizations of subsheaves.

Quite recently, Hashimoto and Keller~\cite{HK19,HK20} have given a new proof of the implication~(3)$\Rightarrow$(1), and a conditional new proof of (3)$\Rightarrow$(1), both in the polarized case when $\omega\in c_1(L)$, for an ample line bundle $L$. Like ours, their approach is variational in nature, but uses geodesics in the space of Hermitian norms on global sections of $E\otimes L^k$ for $k\gg0$.

There has also been a great deal of work on singular versions of Theorem \ref{DUY}. In \cite{BS94}, Bando and Siu introduced the notion of an admissible Hermite-Einstein metric on a torsion-free sheaf, and showed that a reflexive sheaf on a K\"ahler manifold admits an admissible Hermite-Einstein metric if and only if it is polystable. Subsequent work has focused on similar results on singular varieties, see e.g. Chen and Wentworth \cite{CW21}.

\cite{BS94} also introduced a regularization procedure for holomorphic subsheaves of $E$, which was elaborated upon by Jacobs \cite{Jac14} and Sibley \cite{Sib15} (see also \cite{Buc99}). This procedure is an important tool in our proofs (c.f. Theorem \ref{sheaf_slopes}). It was used by Jacobs \cite{Jac14} to generalize Theorem \ref{DUY} to semi-stable bundles, motivated by work of Kobayashi \cite{Kob87}.

Other generalizations of Theorem \ref{DUY} include generalizations to compact Hermitian manifolds (\cite{Buc99, LY87}), and very recent work of Feng-Liu-Wan \cite{FLW18}, which expanded Theorem \ref{DUY} to include the existence of Finsler-Einstein metrics.

\medskip 

\noindent {\bf Organization:} In Section \ref{sec:background} we set some definitions and collect several background results. In Section \ref{sec:filtrations} we prove Theorem A (c.f. Theorem \ref{sheaf_slopes}). In Section \ref{sec:Don} we prove Proposition \ref{propReverseSobolev}, which can be seen as a reverse Sobolev inequality for $w\in\mathcal{S}$, and show the lower semicontinuity of $\mathcal{M}$ on $\mathcal{H}^{1,p}$. We then show Theorems B and C in Section \ref{sec:final}.

\medskip

\subsection*{Acknowledgements} We would like to thank Yoshinori Hashimoto for pointing out an inaccuracy in an earlier draft. The first author was supported by NSF grants DMS-1900025 and DMS-2154380.


\section{Background Material}\label{sec:background}


\subsection{Slope stability and Sobolev Endomorphisms}\label{sec:sobenv}

For any holomorphic, torsion-free sheaf $\mathcal{E}$ on $X$, write:
\[
\mu_\mathcal{E} := \frac{\int_X c_1(\mathcal{E})\wedge \omega^{n-1}}{\mathrm{rk}(\mathcal{E})},
\]
for the {\em slope} of $\mathcal{E}$, with respect to $\omega$. We say $E$ is {\em slope stable} if:
\[
\mu_\mathcal{E} < \mu_E
\]
for all proper, saturated torsion-free subsheaves $\mathcal{E}\subset E$. $E$ is said to be {\em  slope semi-stable} if $\mu_\mathcal{E}\leq \mu_E$ for all such $\mathcal{E}$, and {\em slope unstable} otherwise. A subsheaf satisfying $\mu_\mathcal{E} > \mu_E$ is said to be {\em destabilizing.}

Fix a Hermitian metric $h_0$ on $E$. We can identify $\mathrm{Herm}(E)$, the space of all smooth Hermitian metrics on $E$, with $\ti{\mathcal{H}}$, the space of $h_0$-self adjoint endomorphisms of $E$ by:
\[
h\in\mathrm{Herm}(E) \mapsto \log(hh_0^{-1});
\]
note that geodesics in $\mathrm{Herm}(E)$ map to straight line segements in $\ti{\mathcal{H}}$, and visa versa. We have $\ti{\mathcal{H}} = \mathcal{H}\oplus \R$, where $\mathcal{H}$ is the (geodesically complete) subspace of trace-free endomorphims. 

Write $\mathcal{H}^{1,p} := \mathcal{H}\otimes_{C^\infty} W^{1,p}$, for any $p\geq 1$. We refer to straight line segments in $\mathcal{H}^{1,p}$ as {\em weak} geodesics, sometimes without the adjective if the lack of regularity is clear from context. For any $w \in \mathcal{H}^{1,p}$, we define $\| w \|_L^p$ to be the $L^p$-norm of the operator norm of $w$ i.e.~
\[ \|w\|_{L^p}^p = \int_X \mathrm{tr}(ww^*)^{p/2} \omega^n. \]

The duality pairing between $\mathcal{H}^{1,p}$ and $\mathcal{H}^{1,q}$, $q :=\frac{p}{p-1}$, is defined to be:
\[
\langle w, u\rangle := \int_X \mathrm{tr}(w \ov{u})\omega^n + \sqrt{-1}\int_X \mathrm{tr}(D' w\wedge \pbar u)\wedge \omega^{n-1},
\]
where $D = D' + \pbar$ is the Chern connection of $h_0$; it follows that $\mathcal{H}^{1, q}$, is linearly dual to $\mathcal{H}^{1,p}$. Standard functional analysis implies $\mathcal{H}^{1,p}$ is reflexive for $p > 1$, so by the Banach--Alaoglu Theorem, any bounded subset of $\mathcal{H}^{1,p}$ is weakly compact, i.e. if $\norm{\pbar w_i}_{L^{p}} \leq C$, then there exists a subsequence of the $i$ such that, after relabeling:
\[
w_i \rightharpoonup w\in \mathcal{H}^{1,p},
\]
in the sense that:
\[
\langle w_i, u\rangle \rightarrow \langle w, u\rangle,
\]
for every $u\in\mathcal{H}^{1,q}$.

For the purposes of this paper, it will be convenient to fix $1 < p\leq p_{\max} := \frac{2n}{2n-1}$. Then $p < 2$, and the Sobolev conjugate of $p$ is $p' := \frac{2np}{n-p}$ and:
\[
p' > \frac{p}{2-p} =: p^*
\]
Note that $p'_{\max} = p^*_{\max}$, The Gagliardo-Nirenberg-Sobolev inequality can be given as:
\begin{equation}\label{Sobolev}
\norm{w}_{L^{p'}} \leq C_{\mathrm{Sob}}\norm{w}_{W^{1,p}}\text{ for any }w\in\mathcal{H}^{1,p},
\end{equation}
and by the Sobolev embedding theorem, the inclusion $W^{1,p}\hookrightarrow L^{p^*}$ is compact -- this will be the only reason we need to restrict to $p < p_{\max}$ in the proof of Theorem~B.

Given $w\in \mathcal{H}$, write $\lambda_1 \geq \ldots \geq \lambda_r$ for the eigenvalues of $w$; these will be Lipschitz functions on $X$ such that $\sum_{i=1}^r \lambda_i = 0$. When $w\in\mathcal{H}^{1,p}$, the $\lambda_i$ may only be in $L^{p'}$. 

For any $w\in\mathcal{H}^{1,p}$, we define $e^w$ via the power series formula. The resulting self-adjoint endomorphism will be measurable and a.e. finite, but may not be integrable, and it follows that the same applies to $e^w h_0$.

\subsection{Donaldson Functional and Properness}

Recall now the Donaldson functional; if $w\in\mathcal{H}$, the functional was originally defined by Donaldson \cite{Don87} as:
\[
\mathcal{M}(w) := \mathcal{M}(e^wh_0, h_0) = \int_0^1 \int_X \mathrm{tr}(v_s \Theta_s) \wedge \omega^{n-1}\wedge ds,
\]
where $h_s := e^{sw}h_0$, $v_s := (d_s h_s)h_s^{-1}$, and $\Theta_s := \sqrt{-1}\pbar((\p h_s)h_s^{-1})$ is the curvature of $h_s$ (note the $\sqrt{-1}$ factor). Recall that $\mathcal{M}(tw)$ is convex in $t$, and $\mathcal{M}(0) = 0$. By construction, $\mathcal{M}$ is a Lagrangian for the Hermite-Einstein equation, and it is standard to check that if $w\in\mathcal{H}$ is a minimizer of $\mathcal{M}$ and $\Theta$ is the curvature of $h = e^wh_0$, then:
\[
\Theta\wedge \omega^{n-1} = \gamma I_E \cdot \omega^n,
\]
with $\gamma = \frac{2\pi\, n\mu_E}{\int_X \omega^n}$, i.e. $h$ is Hermite-Einstein.

There is another formulation of the Donaldson functional in terms of the eigenvalues of $w$. For a point $z \in X$, let $\{e_i\}$ be a unitary frame for $E_{z}$ which diagonalizes $w$ at $z$, with $w(z) = \mathrm{Diag}(\lambda_1\mathrm{Id}_{r_1},\dots,\lambda_m\mathrm{Id}_{r_m})$ such that $\lambda_1 > \dots > \lambda_m$ and $\sum_ir_i = r$. Write $\eta^i_j \in \mathcal{A}^{0,1}(X)$  for the $r_i \times r_j$ block matrix of $(\pbar w)(z)$ with respect to the $\{e_i\}$-basis, and define:
\begin{align*}
  \abs{\eta^i_j}^2(z) &:=  n  \frac{\sqrt{-1} \mathrm{tr}(\eta^i_j \wedge ({\eta}^i_j)^{*}) \wedge \omega^{n-1}}{\omega^n}(z) \\
  &= n \sum_{\substack{1 \leq k \leq r_i\\1 \leq l \leq r_j}} \frac{\sqrt{-1}(\eta^i_j)_l^k\wedge \ov{({\eta}^i_j)^k_l} \wedge \omega^{n-1}}{\omega^n}(z),
\end{align*}
where the $(\eta^i_j)_l^k$ are the components of the block matrix $\eta_j^i$. 
We then define the function:

\begin{equation}\label{function_f}
 f_w(z) := \sum_{i,j=1}^m |\eta^i_j|^2\frac{e^{\lambda_i - \lambda_j}  - (\lambda_i - \lambda_j) - 1}{(\lambda_i-\lambda_j)^2},
\end{equation}
where we interpret $\frac{e^x - x - 1}{x^2}$ as $\frac{1}{2}$ when $x = 0$.

\begin{proposition}
$f_w$ is a well-defined function on $X$.
\end{proposition}
\begin{proof}
  Fix $z \in X$ and suppose we have another unitary frame $\{f_i\}$ which diagonalizes $w(z)$. Since the eigenvalues of $w(z)$ are independent of the choice of a  basis, the matrix of $w(z)$ with respect to the $\{f_i\}$-basis will be still be $\mathrm{Diag}(\lambda_1\mathrm{Id}_{r_1},\dots,\lambda_m\mathrm{Id}_{r_m})$, and we can relate the $\{e_i\}$ and $\{f_i\}$ bases by a block diagonal unitary transformation $\mathrm{Diag}(A_1,\dots,A_m)$, where each $A_i$ is an $r_i \times r_i$ unitary matrix.
  
  It follows that $(\dbar w)(z)$ with respect to the $\{f_i\}$-basis is given by $A_i \eta_j^{i} A_j^{*}$, and so it suffices to check that $\mathrm{tr}(\eta^i_j \wedge (\eta^i_j)^{*}) = \mathrm{tr}(A_i \eta_j^{i} A_j^{*} \wedge (A_i \eta_j^{i} A_j^{*})^{*} )$. This can be seen by computing:
 \begin{align*}
   \mathrm{tr}(A_i \eta_j^{i} A_j^{*} \wedge (A_i \eta_j^{i} A_j^{*})^{*} ) &=  \mathrm{tr}(A_i \eta_j^{i} A_j^{*} \wedge A_j (\eta_j^{i})^{*} A_i^{*} ) \\
                                                                            &= \mathrm{tr}(A_i \eta_j^{i} \wedge (\eta_j^{i})^{*} A_i^{*} ) \\
                                                                            &= \mathrm{tr}(A_i^{*} A_i \eta_j^{i} \wedge (\eta_j^{i})^{*}) \\
   &= \mathrm{tr}(\eta_j^{i} \wedge (\eta_j^{i})^{*}),
 \end{align*}
 since each $A_i$ and $A_j$ are unitary.
\end{proof}

\begin{remark}
Note that in the above definition the ranks of the block matrices, $r_i$ could vary with $z \in X$. However, if the eigenvalues of $w$ are constant over $X$, then the ranks $r_i$ do not depend on the points $z \in X$. If $U \subset X$ is such that $E|_U$ is a trivial vector bundle, we can also think of $\eta$ as matrix-valued functions on $U$.
\end{remark}

Write $\Theta_0$ for the curvature of $h_0$. 
\begin{proposition}
For any $w\in\mathcal{H}$, we have:
\[
\mathcal{M}(w) = \int_X f_w  \omega^n + \int_X \mathrm{tr}(\Theta_0 w)\wedge\omega^{n-1}.
\]
\end{proposition}
\begin{proof}
The Donaldson functional satisfies the following \cite[Equations 6.3.27 and 6.3.33]{Kob87}: 
\begin{align*}
  &\mathcal{M}(0) = 0 \\
  &\frac{d\mathcal{M}}{dt}\bigg|_{t = 0} = \int_X \mathrm{tr}(w \Theta_0) \wedge \omega^{n-1} \\
  &\frac{d^2 \mathcal{M}}{dt^2} = \int_X | \dbar w |_{h_t}^2 \omega^{n-1} \\
  &\quad \quad \ = \sum_{1 \leq i,j \leq m}\int_X e^{t(\lambda_i - \lambda_j)}| \eta_j^i |^2  \omega^n.
\end{align*}    
For $z \in X$, consider the following system of ODEs, for some $\phi_z : \R_{\geq 0} \to \R $.
\begin{align*}
  \phi_z(0) &= 0 \\
  \phi_z'(0) &= \frac{\mathrm{tr}(w(z) \cdot \Theta_0(z)) \wedge \omega^{n-1}}{\omega^n} \\
  \phi_z''(t) & =\sum_{i,j} e^{t(\lambda_i - \lambda_j)}| \eta_j^i |^2. 
\end{align*}
It is easy to see that the solution to this system is given by \[\phi_z(t) = f_{tw}(z) + \frac{\mathrm{tr}(w(z) \cdot \Theta_0(z)) \wedge \omega^{n-1}}{\omega^n}t.\]
Furthermore, we have that $\mathcal{M}(tw) = \int_X \phi_z(t) \omega^n$. Setting $t=1$, we get the required result. 
\end{proof}

The function $f_w$ makes sense for any $w\in \mathcal{H}^{1,1}\supset \mathcal{H}^{1,p}$, and is always $\geq 0$. This lets us extend the definition of $\mathcal{M}$ as:
\begin{definition}
Given $w\in \mathcal{H}^{1,1}$, we define:
\[
\mathcal{M}(w) := \int_X f_w \omega^n + \int_X \mathrm{tr}(\Theta_0 w)\wedge\omega^{n-1},
\]
if $f_w \in L^{1}(X)$; otherwise, set $\mathcal{M}(w) := +\infty$. We define:
\[
\mathcal{S} := \{w\in \mathcal{H}^{1,1}\ |\ \mathcal{M}(w) < \infty\}.
\]
\end{definition}
As already remarked in the introduction, it will follow from Proposition \ref{propReverseSobolev} that $\mathcal{S}\subset \mathcal{H}^{1,p_{\max}}$.

\begin{proposition}\label{convex}
Suppose that $w\in \mathcal{H}^{1,1}$. Then $\mathcal{M}$ is convex on the weak geodesic ray $(tw)_{t\geq 0}$.
\end{proposition}
\begin{proof}

Note that
  \[
 f_{tw}(z) := \sum_{i,j=1}^m |\eta^i_j|^2\frac{e^{t(\lambda_i - \lambda_j)}  - t(\lambda_i - \lambda_j) - 1}{(\lambda_i-\lambda_j)^2}.
\]
The convexity of $e^x - 1 - x$ implies that for $t, s \in \R_{\geq 0}$ and $\alpha\in (0,1)$ we have
\[
  f_{(\alpha s + (1-\alpha) t)w} \leq \alpha f_{sw} + (1-\alpha)f_{tw}. 
\]
Thus, we get that
\begin{align*}
  \mathcal{M}\left((\alpha s + (1-\alpha) t) w \right) &= \int_X f_{(\alpha s + (1-\alpha) t)w} \omega^n + (\alpha s + (1-\alpha)t) \int_X \mathrm{tr}(\Theta_0  w)\wedge\omega^{n-1} \\
&\leq \alpha\mathcal{M}(sw) + (1-\alpha)\mathcal{M}(tw). 
\end{align*}

\end{proof}

\begin{definition}
We say that $\mathcal{M}$ is {\em proper} on $\mathcal{S}$ with respect to the $W^{1,p}$-norm if there exists a constant $C > 0$ such that:
\[
\mathcal{M}(w) \geq C^{-1}\norm{w}_{W^{1,p}} - C
\]
for all $w\in \mathcal{S}$.
\end{definition}

If $\mathcal{M}$ is proper with respect to $W^{1,p}$, then clearly it is proper with respect to $W^{1, q}$ for any $1 \leq q < p$.

By Proposition \ref{convex}, $\mathcal{M}$ cannot be bounded below if there exists a weak geodesic ray $(tw)_{t \geq 0}$ along which $\mathcal{M}$ has negative asympotitc slope, i.e.
\[
\lim_{t\rightarrow\infty} \frac{d}{ds}\bigg|_{s=t}\mathcal{M}(sw) < 0.
\] 
Similarly, $\mathcal{M}$ cannot be proper if there exists a weak geodesic ray along which $\mathcal{M}$ has non-positive asymptotic slope.

\subsection{Weakly Holomorphic Projections}

Finally, we recall the weak holomorphicity theorem of Uhlenbeck-Yau (see also \cite{Pop03}):

\begin{theorem}\label{UY}
(Uhlenbeck-Yau \cite{UY86, UY89}) Suppose that $\pi\in \ti{\mathcal{H}}^{1,2}$ is a weakly holomorphic projection, i.e.:
\[
\pi^2 = \pi\quad\text{ and }\quad (I - \pi)\pbar\pi = 0.
\]
Then $\pi$ is actually the projection onto a holomorphic subsheaf of $E$.
\end{theorem}


\section{Geodesic Rays Associated to a Filtration}\label{sec:filtrations}


In this section, we show how to construct a natural geodesic ray from a filtration of $E$ and a strictly decreasing set of real numbers $\lambda_1 > \ldots > \lambda_m$; the asymptotic slope of $\mathcal{M}$ along the resulting ray will be computed by the slopes of the filtration, weighted by the $\lambda_i$. 

\subsection{The case of a subbundle filtration}

Fix a Hermitian metric $h_0$ on $E$. Suppose first that we have a filtration of $E$ by holomorphic subbundles $E = E_1 \supset \ldots \supset E_m\supset E_{m+1} = \{0\}$. Let $F_i := E_i/ E_{i+1}$; we have a smooth, orthogonal splitting $E = \bigoplus_{i=1}^m G_i$, such that $G_i \cong F_i$ for each $i=1, \ldots, m-1$, and $G_m = E_m$.

Let $r_i := \mathrm{rk}(F_i)$ and $s_k = \sum_{i=k}^m r_k = \mathrm{rk}(E_k)$. Consider $\{e_i\}_{i=1}^r$, a local holomorphic frame of $E$, such that $\{e_i\}_{i=r - s_k + 1}^r$ frames $E_k$ for each $k = 1, \ldots, m$. Write $\{f_{i}\}_{j=1}^{r}$ for the orthogonal projections of the $\{e_i\}_{i = 1}^{r}$ onto the decomposition $\bigoplus_{i=1}^m G_i$ -- for each $k = 1, \ldots, m$, $\{f_i\}_{i=r-s_k + 1}^{r - s_{k-1}}$ is a (smooth) frame for $G_k$ (note $e_i = f_i$ for $r-s_{m-1}+1 < i\leq r$).

Let $\lambda_1 > \ldots > \lambda_m$ and set $\delta := \min_{i < j} \{ \lambda_i - \lambda_j\} > 0$. Define $w\in \ti{\mathcal{H}}$ by $w(v) = e^{\lambda_i}v$ for all $v\in G_i$, and $h_t := e^{tw}h_0$, $t\geq 0$; by definition, this is a geodesic ray in $\ti{\mathcal{H}}$.

{
\begin{lemma}
\label{lemChernConnectionH0}
Let $D_{0}$ denote the Chern connection of $h_0$. The connection form $\alpha_0$ for $D_{0}$ can be written with respect to the $\{f_i\}$-basis as
  \[ \alpha_0 = 
\begin{pmatrix}
  \beta_{11} \dots \beta_{1m} \\
  \vdots \ddots \vdots \\
  \beta_{m1} \dots \beta_{mm}
\end{pmatrix}.
\]
Using the identification $G_i \simeq F_i$, the $\beta_{ij}$ have the following description. 
\begin{itemize}
\item $\beta_{ii} = \alpha_{F_i,0}$, where $\alpha_{F_i,0}$ is the connection form for $D_{F_i,0}$, the Chern connection of the Hermitian metric induced by $h_{0}$ on $F_i$. 
\item For $i > j$, $\beta_{ij}$ is a $\Hom(F_i, F_j)$-valued $(1,0)$-form (i.e. an $r_i\times r_j$ matrix of $(1,0)$-forms). For $i < j$, $\beta_{ij}$ is a $\Hom(F_j, F_i)$-valued $(0,1)$-form.
\item $\beta_{ji}$ is the $h_0$-adjoint of $-\beta_{ij}$ for $i \neq j$ i.e.~
  \[
h_0(\beta_{ij}v,w) + h_0(v,\beta_{ji}w) = 0 \text{ for all } v \in F_i, w \in F_j.
  \]
\end{itemize}
\end{lemma}
\begin{proof}
  The first and second part of the claim can by easily proved by induction on $m$. In the case of $m = 2$, the proof can be found in \cite[Propositions 1.6.4-1.6.6]{Kob87}.

  To see that last part of the claim, let $v \in F_i$ and $w \in F_j$. By abuse of notation, let $v, w$ also denote their respective images in $E$ under the identifications $F_i \simeq G_i$ and $F_j \simeq G_j$. Then, $h_0(v,w) = 0$ and we get
  \begin{align*}
    dh_0(v,w) &= h_0(D_{0}v,w) + h_0(v,D_{0}w) \\
    0 &= \sum_k h_0(\beta_{ik}v,w) + \sum_l h_0(v, \beta_{jl}w) \\
    0 &= h_0(\beta_{ij}v,w) + h_0(v,\beta_{ji}w). 
  \end{align*}
\end{proof}

  \begin{lemma}
 \label{lemChernConnectionHt}
  Let $\alpha_t$ be the connection form of $D_t$, the Chern connection of $h_t$. Then, the  matrix of $\alpha_t$ can be written as 
  $\alpha_t = \left( \beta_{t,ij} \right)$ with respect to the $\{f_i\}$-basis, where,
  \[
    \beta_{t,ij} =
    \begin{cases}
      e^{-t(\lambda_j - \lambda_i)} \beta_{ij} &\text{ if } i > j \\
\beta_{ij} \text{ otherwise,}
    \end{cases}
  \]
and $\alpha_0 = (\beta_{ij})$, as in Lemma \ref{lemChernConnectionH0}.
\end{lemma}
\begin{proof}
  It is enough to check that $D_t := d + \alpha_t$ satisfies the properties of the Chern connection. First we check that $D_t'' = \dbar_E$. It is enough to check that if $v \in G_i$, then $D''_tv = \dbar_E v$. We know that $D_t''v$ is the $(0,1)$ part of $D_tv = dv + \sum_j \beta_{t,ij} v$. Using Lemma \ref{lemChernConnectionH0}, we get that
\begin{align*}
  D_t''v &= D_{F_i}'' v + \sum_{j > i} \beta_{t,ij}v \\
         &= D_{F_i}'' v + \sum_{j > i} \beta_{ij}v \\
  &= D_0''v.
\end{align*}
Since $D_0$ is the Chern connection on $E$ with respect to $h_0$, we have that $D_t'' v = D_0''v = \dbar_E v$.

To show that $D_t$ is compatible with $h_t$, we need to show that
\[
d h_t(v,w) = h_t(D_tv,w) + h_t(v,D_tw), 
\]
for $v,w \in E$. It is enough show this in the case when $v \in G_i$ and $w \in G_j$. If $i = j$, then
\begin{align*}
  d h_t(v,w) &= e^{t\lambda_i} d h_0(v,w) \\
             &= e^{t\lambda_i} \left( h_0(D_0v,w) + h_0(v,D_0w) \right) \\
             &= e^{t\lambda_i} \left( h_0(D_{F_i}v,w) + h_0(v,D_{F_i}w) \right) \\
             &= h_t(D_tv,w) + h_t(v,D_tw). 
\end{align*}
Now consider the case of $i \neq j$. Without loss of generality, assume that $i > j$. We have that $d h_t(v,w) = 0$ and the right hand side is given by
\begin{align*}
  h_t(D_tv,w) + h_t(v,D_tw) &= h_t(\beta_{t,ij}v,w) + h_t(v,\beta_{ji}w) \\
                            &= e^{t\lambda_j}h_0(e^{-t(\lambda_j-\lambda_i)} \beta_{ij}v,w) + e^{t\lambda_i} h_0(v, \beta_{ji}w) \\
                            &= e^{t\lambda_i} (h_0(\beta_{ij}v,w) + h_0(v,\beta_{ji}w)) \\
  &= 0.
\end{align*}
where the last equality follows from Lemma \ref{lemChernConnectionH0}. 
\end{proof}

\begin{theorem}\label{thmMAlongGeodesic}
  \label{slope}
Along the geodesic ray $h_t$, the Donaldson functional is given by
\begin{align*}
  \mathcal{M}(tw) = 2\pi\sum_{i=1}^m \lambda_i \mathrm{rk}(F_i) (\mu_{F_i} - \mu_E)t - \sum_{1 \leq i < j \leq m} B_{ji} (1 - e^{-t(\lambda_i - \lambda_j)}),
\end{align*}
where $B_{ji} = \int_{X}| \beta_{ji} |_{h_0}^2 \omega^n$ is a non-negative integer. 
\end{theorem}
\begin{proof}
  Write $\Theta_t$ for the curvature of $h_t$. In the $\{f_i\}$-basis, we can write $\Theta_t = (\theta_{t,ij})$, for some $\Hom(F_i, F_j)$-valued $(1,1)$-forms $\theta_{t,ij}$.
  
  Recall that the Donaldson functional along a geodesic is given by
  \[
\mathcal{M}(tw) = \int_0^t \int_X \mathrm{tr}(w \Theta_s) \wedge \omega^{n-1} ds - c \int_X \mathrm{tr}(w) \omega^n. 
  \]
   Since the matrix for $w$ in the $f$-basis is  $\mathrm{Diag}(\lambda_1\mathrm{Id}_{r_1}, \ldots, \lambda_m\mathrm{Id}_{r_m})$, we can locally write:
  \begin{equation}\label{eqnhellothere}
     \int_0^t \mathrm{tr}(w\Theta_s) ds = \sum_i \lambda_i \int_0^t \mathrm{tr}(\theta_{s,ii}). 
  \end{equation}
  Using Lemma \ref{lemChernConnectionHt}, we can compute $\theta_{s,ii}$ by using the local expression for the curvature $\Theta_s = d\alpha_s + \alpha_s\wedge\alpha_s$, which gives:
  \begin{align*}
    \theta_{s,ii} &= \Theta_{i} + \sqrt{-1}\sum_{i\not= j} e^{-s|\lambda_i - \lambda_j|}\beta_{ji}\wedge\beta_{ij}, 
  \end{align*}
  writing $\Theta_i$ for the curvature of the metric induced on $F_i$ by $h_0$. We can now integrate \eqref{eqnhellothere} in $s$ to get:
  \begin{align*}
  \int_0^t \mathrm{tr}(w \Theta_s) ds
  &= \sum_i t\lambda_i \mathrm{tr} (\Theta_{i}) + \sqrt{-1}\sum_{i \neq j} \frac{\lambda_i (1-e^{-t|\lambda_i-\lambda_j|}) }{|\lambda_i - \lambda_j|} \mathrm{tr}(\beta_{ji}\wedge \beta_{ij}).
\end{align*}
   This can be simplified by using the relations $\beta_{ij} = -\beta_{ji}^*$ and $\lambda_j < \lambda_i$ for $i < j$ to give:
\begin{equation}\label{eqnpointless}
  \int_0^t \mathrm{tr}(w \Theta_s) ds
  = t\sum_i \lambda_i \mathrm{tr} (\Theta_{i}) - \sqrt{-1}\sum_{i < j} (1-e^{-t(\lambda_i-\lambda_j)}) \mathrm{tr}(\beta_{ji}\wedge \beta_{ji}^*).
\end{equation}

Now the second term in $\mathcal{M}(tw)$ is given by
\begin{align*}
  \gamma\cdot t\int_X \mathrm{tr}(w) \omega^n &= \gamma \cdot t \sum_i \lambda_i r_i \int_X\omega^n \\
  &= 2 \pi \lambda_i r_i \mu_E t,
\end{align*}
so by integrating \eqref{eqnpointless} over $X$ and combining the two expressions, we get the result. 
\end{proof}

}
\begin{corollary}\label{2stage_filtration}
Consider a two-step filtration given by a subbundle $E_2 \subset E_1 = E$.Set $\lambda_1 = 1, \lambda_2 = 0,$ and $F = E_1/E_2$, and let $h_t$ be the geodesic ray constructed in Theorem \ref{slope} from this data. Then:
\[
\mathcal{M}(h_t) = 2\pi \mathrm{rk}(F) (\mu_F - \mu_E)t - B (1-e^{-t}),
\]
for some non-negative constant $B$ depending only on $h_0$. It follows that:
\[
\lim_{t\rightarrow\infty} \frac{d}{ds}\Big|_{s=t} \mathcal{M}(h_s) = 2\pi \mathrm{rk}(F) (\mu_{F} - \mu_{E}).
\]

We conclude that if $E_2$ is a destabilizing  subbundle of $E$, then $\mathcal{M}$ has negative slope along $h_t$; if $\mu_{E_2} = \mu_E$, then $\mathcal{M}$ is monotonically decreasing along $h_t$, with asymptotic slope 0. 

Moreover, we see that $\mathcal{M}$ is zero along $h_t$ if and only if there exists a holomorphic splitting $E = E_2 \oplus E_2^{\perp_{h_0}}$ with $\mu_{E_2} = \mu_E$.
\end{corollary}

\begin{proof}
  The expression for $\mathcal{M}(h_t)$ and the asymptotic slope follow from Theorem \ref{thmMAlongGeodesic}. If $E_2$ is destabilizing, then $\mu_{E_2} > \mu_E$ and $\mu_F < \mu_E$, and the asymptotic slope is negative. If $\mu_{E_2} = \mu_E$, then $\mu_E = \mu_F$ and $\mathcal{M}(h_t) = -B(1-e^{-t})$ is decreasing in $t$ as $B \geq 0$.

 Finally, since $B = \| \beta \|_{h_0}^2$, where $\beta$ is the second fundamental form of $E_2$ in $(E,h_0)$, we have $B = 0$ if and only if there is a holomorphic splitting $E = E_2 \oplus E_2^{\perp_{h_0}}$ (see \cite[Proposition 1.6.14]{Kob87}). If $E = E_2 \oplus E_{2}^{\perp_{h_0}}$ is a holomorphic splitting and $\mu_{E_2} = \mu_E$, then $\mu_E = \mu_F$ and $B = 0$ and $\mathcal{M}$ is zero along $h_t$. Conversely, if $\mathcal{M}$ is zero along $h_t$, then $\mu_E = \mu_F$ and $B = 0$, which implies that $E = E_2 \oplus E_2^{\perp_{h_0}}$ is a holomorphic splitting with $\mu_E = \mu_{E_2}$. 
  
\end{proof}

We also have:
\begin{corollary}
  \label{corNegSlope}
Suppose that $\{E_i\}_{i=1}^m$ is a filtration of $E$, and $h_t$ is a geodesic of the form constructed in Theorem \ref{slope}, for some $\lambda_1 > \ldots \lambda_m > 0$. If $\mathcal{M}(h_t)$ has non-positive asymptotic slope, then at least one of the $E_i$ is such that $\mu_{E_i} \geq \mu_E$.
\end{corollary}
\begin{proof}
  Using Lemma \ref{slope}, we have that
  \[
\sum_i \lambda_i r_i (\mu_{F_i} - \mu_E) \leq 0. 
\]
We need to show that $\mu_{E_i} \geq \mu_E$ for some $2 \leq i \leq m$. Note that $\deg(E_i) = \sum_{j=i}^m \deg(F_i) = \sum_{j=i}^m r_i \mu_{F_i}$ i.e.~we need to show that
\begin{align*}
\frac{r_i \mu_{F_i}+ r_{i+1} \mu_{F_{i+1}} + \dots r_m \mu_{F_m}}{r_i + r_{i+1} + \dots r_m} \geq \mu_{E}. 
\end{align*}

Rearranging, we see that it is enough to show that
\[
 \sum_{j=i}^m r_j(\mu_{F_j}- \mu_E) \geq 0. 
\]
for some $2 \leq i \leq m$. Applying Lemma \ref{lemAlphaSumPositive} with $a_i = r_i(\mu_{F_i} - \mu_E)$, we get the required result.
\end{proof}

\begin{lemma}
\label{lemAlphaSumPositive}
Let $\lambda_1,\ldots,\lambda_m, a_1,\ldots,a_m \in \R$ such that $\lambda_1 > \ldots > \lambda_m$, $\sum_i a_i = 0$, and $\sum_i \lambda_i a_i \leq 0$ (respectively $< 0$). Then, there exist  $2 \leq i \leq m$ such that $a_i + \dots  + a_m \geq 0$ (respectively $ > 0$). 
\end{lemma}
\begin{proof}
  We prove the result by induction on $m$ for the non-strict inequality. The case of the strict inequality follows similarly. First consider the base case when $m = 2$. Since replacing $(\lambda_1,\lambda_2)$ by $(\lambda_1+c,\lambda_2+c)$ for $c \in \R$ leaves the sum $\lambda_1 a_1 + \lambda_2 a_2$ unchanged, we may assume that $\lambda_1 + \lambda_2 = 0$ i.e.~$\lambda_2 = -\lambda_1$ and $\lambda_1$ is positive. Since $a_1 = - a_2$, we have $-2\lambda_1 a_2 = \lambda_1 a_1 + \lambda_2 a_2 \leq 0$, giving $a_2 \geq 0$.

  Now consider the case when $m > 2$. If $a_m \geq 0$, we pick $i = m$ and we are done. So assume that $a_m < 0$. Since $\lambda_{m-1} > \lambda_m$, we have that $\lambda_{m-1} a_m < \lambda_m a_m$. Then,
  $$0 \geq \sum_{i=1}^m \lambda_i a_{i} \geq \lambda_1 a_1 + \ldots + \lambda_{m-2} a_{m-2} + \lambda_{m-1}( a_{m-1} + a_m).$$
  Applying induction hypothesis to $(\lambda_1,\dots,\lambda_{m-1}), (a_1,\dots,a_{m-2},a_{m-1}+a_m)$, we get the claim. 
\end{proof}

\subsection{General Filtrations}

Consider a general filtration of saturated holomorphic subsheaves $E = \mathcal{E}_1 \supset \mathcal{E}_2 \supset \ldots \supset \mathcal{E}_m \supset \{0\}$. Again let $\lambda_1 > \ldots > \lambda_m$ and $\mathcal{F}_i = \mathcal{E}_i/\mathcal{E}_{i+1}$. 

By regularizing the $\mathcal{E}_i$, as in \cite[Proposition 4.3]{Sib15}, we can construct a geodesic ray as in Theorem \ref{slope} on a proper modification of $X$:
\begin{theorem}\label{sheaf_slopes}
{
  There exist $w \in \mathcal{H}^{1,p}$ such that
  \begin{align*}
  \mathcal{M}(tw) = 2\pi\sum_{i=1}^m \lambda_i \mathrm{rk}(\mathcal{F}_i) (\mu_{\mathcal{F}_i} - \mu_E)t - \sum_{1 \leq i < j \leq m} B_{ji} (1 - e^{-t(\lambda_i - \lambda_j)}),
\end{align*}
where $B_{ji}$ is a non-negative integer. 
}
\end{theorem}
Note that Theorem A is a direct corollary of this result.
\begin{proof}
{
  Using \cite[Proposition 4.3]{Sib15}, we can find a proper modification $f: \tilde{X} \to X$ such that the saturation of $f^{*}\mathcal{E}_i$ in $f^{*}E$ is a sub-bundle for all $i$. Let $E_i$ denote the saturation of $f^{*}\mathcal{E}_i$ in $f^{*}E$ and let $F_i$ denote the quotient bundle $f^{*}E/E_i$. Applying the construction of Theorem \ref{thmMAlongGeodesic}, we can find a smooth endomorphism $\tilde{w} \in \mathcal{H}_{\tilde{X}}$ such that
  \begin{align*}
  \mathcal{M}_{\tilde{X}}(t\tilde{w}) = 2\pi\sum_{i=1}^m \lambda_i \mathrm{rk}(F_i) (\mu_{F_i} - \mu_E)t - \sum_{1 \leq i < j \leq m} B_{ji} (1 - e^{-t(\lambda_i - \lambda_j)}).
  \end{align*}
  where $\mathcal{M}_{\tilde{X}}$ is the Donaldson functional on $\tilde{X}$ with respect to the reference metric $f^{*}h_0$ and the form $f^{*}\omega$. Note that the form $f^{*}\omega$ is no longer a K\"ahler form on $\tilde{X}$; however we do not require this in the construction done in Theorem \ref{thmMAlongGeodesic}. 

  Since $f$ is smooth and the pushforward on currents commutes with $d$, we see that $w := f_*\ti{w} \in \mathcal{H}^{1,\infty}$.
  To finish the proof, it is enough to show that $\mathcal{M}_{\tilde{X}}(t\tilde{w}) = \mathcal{M}(tw)$ and that $\mu_{F_i} = \mu_{\mathcal{F}_i}$. The first part follows from applying the change of variables formula to the definition of $\mathcal{M}_{\tilde{X}}$ and $\mathcal{M}$.
  To show that $\mu_{F_i} = \mu_{\mathcal{F}_i}$, it is enough to check that $\mu_{E_i} = \mu_{\mathcal{E}_i}$. Since $\mathcal{E}_i$ is torsion-free, $f^{*}\mathcal{E}_i$ and its saturation can only differ along a submanifold of codimension at least 2. Thus $f^{*}\det(\mathcal{E}_i)$ and $\det(E_i)$ are line bundles that differ along a submanifold of codimension at least 2 and therefore must be isomorphic. We have
  \begin{align*}
\mu_{\mathcal{E}_i} = \frac{\deg_{\omega}(\det(\mathcal{E}_i))}{\mathrm{rk}(\mathcal{E}_i)} = \frac{\deg_{f^{*}\omega}(\det (E_{i}))}{\mathrm{rk}(E_i)} = \mu_{E_i}. 
  \end{align*}
}
\end{proof}

The following result follows from Corollary \ref{corNegSlope}. 
\begin{proposition}
Let $w$ be as constructed in Theorem \ref{sheaf_slopes}. If we further have that $\mathcal{M}(tw) \leq 0$ for all $t \geq 0$, then there exists $2 \leq i \leq m$ such that $\mu_{\mathcal{E}_i} \geq \mu_E$. \qed
\end{proposition}


\section{Further Properties of the Donaldson Functional}\label{sec:Don}


Throughout this section and the next, we use $C$ to represent a positive constant, whose exact value may change from line to line, but which can always be computed from fixed data.

\subsection{Reverse Sobolev Inequality}

The following estimate is essentially due to Donaldson; it immediately gives a reverse Sobolev inequality for $w \in \mathcal{S}$,
  showing that $\| w \|_{W^{1,p}}$ is equivalent to $\| w \|_{L^{p^{*}}}$ when both the norms are large.
\begin{proposition}\label{propReverseSobolev}
  Suppose that $w\in\mathcal{H}^{1,p}$. Then there exists a constant $C \geq 0$, depending only on $(X, \omega)$ and $(E, h_0)$, such that:
\[
\norm{w}_{W^{1,p}} \leq C \mathcal{M}(w) + C(\norm{w}_{L^{p^{*}}} + 1).
\]
\end{proposition}
\begin{proof}
Recall that $\mathcal{M}(w) = \int_X f_w \omega^n + \int_X \mathrm{tr}(\Theta_0 w)\wedge\omega^{n-1}$. Therefore, 
\[
\int_X f_w \omega^n \leq \mathcal{M}(w) +  \left|\int_X \mathrm{tr}(\Theta_0 w)\wedge\omega^{n-1}\right|.
\]
Now we find a lower bound for $\int_X f_w \omega^n$. Since 
\[
\frac{e^x - 1 - x}{x^2} \geq \frac{1}{2(|x| + 1)}, 
\]
we have that
$$
\int_X f_w \omega^n \geq \frac{1}{2} \int_X \frac{\sum_{i,j}|\eta_{ij}|^{2}}{|w| + 1}\omega^{n} = \frac{1}{2} \int_X \frac{|\tr{\omega}\dbar w |^{2}}{|w| + 1}\omega^n,$$ 
where $|w|$ and $|\tr{\omega}\dbar w|$ denote the pointwise operator norms of $w$ and $\tr{\omega}\dbar w$. 
Writing $|\tr{\omega}\dbar w|^p   =\frac{|\tr{\omega}\dbar w|^p}{(|w| + 1)^{p/2}} \cdot (|w| + 1)^{p/2}$ and applying the H\"older inequality with conjugate exponents $2/p$ and $2/(2-p)$ gives: 
\[
  \| \tr{\omega}\dbar w \|_{L^p} \leq
  \left\| \frac{|\tr{\omega}\dbar w|}{\sqrt{|w| + 1}} \right\|_{L^2} 
  \cdot \left\||w| + 1 \right\|_{L^{p^{*}}}^{1/2}, 
\]
since $p^* = \frac{p}{2-p}$. It follows that
\begin{align*}
  \int_X f_w \omega^n &\geq \frac{1}{2} \frac{\| \tr{\omega}\dbar w \|^2_{L^p}}{\left\| |w| + 1 \right\|_{L^{p^{*}}}} \geq \frac{1}{2} \frac{\| \dbar w \|^2_{L^p}}{ \|w\|_{L^{p^{*}}} + C}.
\end{align*}

On the other hand, since $h_0$ is a smooth Hermitian metric, we also have that
\[
 \left| \int_X \mathrm{tr}(\Theta_0 w)\wedge\omega^{n-1} \right|   \leq C \| w \|_{L^{p^{*}}}.
\]
Combining the two inequalities above, we get that
\[
\frac{\| \dbar w \|^2_{L^p}}{\|w\|_{L^{p^{*}}} + C} \leq 2(\mathcal{M}(w) +  C\|w\|_{L^{p^{*}}}).
\]
Simplifying, we get that
\begin{align*}
  \| \dbar w \|_{L^p} &\leq \sqrt{2(\mathcal{M}(w) +  C\|w\|_{L^{p^{*}}}) (\| w \|_{L^{p^{*}}} + C)}  \\
  &\leq \frac{\mathcal{M}(w) +  C\|w\|_{L^{p^{*}}}  + C}{\sqrt 2},
\end{align*}
by the AM-GM inequality. Since $p \leq p^*$, we conclude.
\end{proof}

\subsection{Lower-semi-continuity}

\begin{proposition}\label{prop:Mlsc}
The Donaldson functional is lower-semi-continuous on $\mathcal{H}^{1,p}$ with respect to the weak $W^{1,p}$-topology, i.e. if $w_i\rightharpoonup w$ then:
\[
\mathcal{M}(w) \leq \liminf_{k\rightarrow\infty} \mathcal{M}(w_k).
\]
\end{proposition}
\begin{proof}
The proof will rely on several lemmas, whose proofs we provide at the end -- the idea is to establish lower semi-continuity of the $w$ and $\pbar w$ terms in $\mathcal{M}(w)$ separately, and then combine them.

We start with some setup; we can find a subsequence $\{w_{k_\ell}\}$ such that:
\[
\lim_{\ell\rightarrow\infty} \mathcal{M}(w_{k_\ell}) = \liminf_{k\rightarrow\infty} \mathcal{M}(w_k).
\]
By the Sobolev embedding theorem, $w_{k_{\ell}}$ converges to $w$ in $L^p$, so we can extract a further subsequence which converges pointwise a.e.. Hence, by relabeling, we may assume without loss of generality that $w_k\rightarrow w$ pointwise a.e, in $L^p$, and weakly in $W^{1,p}$.

Since $\Theta_0$ is smooth, $\int_X \mathrm{tr}(\Theta_0w_k)\wedge\omega^{n-1}$ is continuous in $k$, so we only need to show that:
\[
\liminf_{k\rightarrow\infty} \int_X f_{w_k}\omega^n \geq \int_X f_w \omega^n.
\]

The problem is local, so fix a trivializing open $U$ for $E$. Let $\{ e_i \}$ be a unitary basis of $E|_U$ that such that, $D$, the matrix of $w$ is diagonal with respect to $\{e_i\}$ and has decreasing entries. We also consider unitary changes of bases $A_k$ which diagonalize $w_k$, so that, in the fixed $\{e_i\}$ basis, $w_k = A_k D_k A^*_k$ for a diagonal matrix $D_k$, whose entries are organized in decreasing order.

We seek to work with pointwise values for $w$, which requires us to deal with eigenvalues with multiplicity. To this end, consider:
\[
\mathcal{R} = \left\{ (r_1,\dots,r_m) \Bigg| \ r_1,\dots,r_m \in \mathbb{Z}_{> 0} \text{ for some } m > 0,  \text{ and } \sum^m_{i=1} r_i = r   \right\},
\]
the set of all partitions of $r$. For each $\ul{r}\in\mathcal{R}$, define $U_{\ul{r}}$ to be the set of all $z\in X$ such that $w$ has eigenvalues $\lambda_{r} > \ldots > \lambda_m$, with multiplicities given by $\ul{r}$, at $z$. Then $U_{\ul{r}}$ is a measurable set and:
\[
U := \bigsqcup_{\ul{r} \in \mathcal{R}} U_{\ul{r}}.
\]

Fix $\ul{r}\in\mathcal{R}$. Then, $D = \mathrm{diag}(\lambda_1\mathrm{Id}_{r_1},\dots,\lambda_m\mathrm{Id}_{r_m})$, on $U_{\ul{r}}$ where $\lambda_1 > \ldots > \lambda_m$, and $\lambda_i: U_{\ul{r}} \to \mathbb{R}_{\geq 0}$ can vary with $z \in U_{\ul{r}}$. 

Pointwise convergence implies that $A_k D_k A^*_k \rightarrow D$. If we write:
 \[
A_k = 
\begin{pmatrix}
  (A_k)^1_1 \dots (A_k)^1_m \\
  \vdots \ddots \vdots \\
  (A_k)^m_1 \dots (A_k)_m^m
\end{pmatrix}
\]
for $r_i\times r_j$-block matrices $(A_k)^i_j$, then we can apply Lemma \ref{lemPointwiseConvergenceOfDN} below to get that $D_k \rightarrow D$ and that each of the $(A_k)_j^i$ is almost unitary, in the sense that:
\begin{itemize}
\item $(A_k)^i_i((A_k)^i_i)^{*} \to \mathrm{Id}_{r_i}$ and $((A_k)^i_i)^{*}(A_k)^i_i\to \mathrm{Id}_{r_i}$
\item $(A_k)^i_j \to 0$ for $i \neq j$.
\end{itemize}

Define $\phi(x) := \frac{e^x - 1 - x}{x^2}$, $x\in\R$, with $\phi(0) := \frac{1}{2}$. Write $\lambda_{k,1} \geq \ldots \geq \lambda_{k,r}$ for the eigenvalues of $w_k$, and define $\Lambda_{k,i,a} := \lambda_{r_1 + \ldots + r_{i-1} + a}$, for any $1\leq i\leq m$ and $1\leq a \leq r_i$. We have that $\Lambda_{k,i,a}\rightarrow \lambda_i$ from the pointwise convergence, and hence $\phi(\Lambda_{k,i,a} - \Lambda_{k,j,b}) \rightarrow\phi(\lambda_i - \lambda_j)$ as $k\rightarrow\infty$. 

Write $\pbar w_k = \eta_k$ in the $A_k$-basis; in the fixed $\{e_i\}$-basis, this becomes:
\[
\ti{\eta}_k = A_k \eta_k A^*_k, 
\]
which by Lemma \ref{lemLInftyBais} weakly converges to $\eta = \pbar w$. We now have that:
\[
f_{w_k} = \sum_{1\leq i,j\leq m} \sum_{\substack{1\leq a\leq r_i \\ 1\leq b\leq r_j}} \phi(\Lambda_{k,i,a} - \Lambda_{k,j,b}) |((\eta_k)_j^i)_b^a|^2
\]
on all of $U_{\ul{r}}$; thus, we may conclude the proof by applying Lemma \ref{lemFatou} with $Y = U_{\ul{r}}$, $A = \{ (a,b) | 1 \leq a \leq r_i, 1 \leq b \leq r_j \}$, $f_{k,\alpha} = \chi_U \cdot \phi(\Lambda_{k,i,a} - \Lambda_{k,j,b})$, $f = \chi_U \cdot \phi(\lambda_i - \lambda_j)$, $\mu_{k,\alpha} = |((\eta_k)^i_j)^a_b |^2$ and $\mu = |\eta^i_j|^2$, provided that we first establish:
\begin{equation}\label{eqnSubGoal}
\liminf_{k\rightarrow\infty} \int_S \abs{(\eta_k)_j^i}^2\omega^n \geq \int_S \abs{\eta_j^i}^2\omega^n
\end{equation}
for every $1\leq i,j \leq m$ and every measurable $S\subset U_{\ul{r}}$.

We show this by observing the the $L^2$ norm is lower semicontinuous with respect to weak $L^1$ convergence, essentially. Start by recalling that:
\[
\int_S \abs{\eta_i^j}^2\omega^n = \sup_{\substack{g\in L^2(S)\cap L^\infty(S)\\ \norm{g}_{L^2(S)} = 1}} \abs{\int_S \mathrm{tr}(\eta_j^ig^*)\omega^n },
\]
where $g$ is a $r_i\times r_j$ matrix valued function. Fix such a $g$, and note that:
\[
\int_S \mathrm{tr}((\ti{\eta}_k)_j^ig^*)\omega^n \rightarrow \int_S \mathrm{tr}(\eta_j^i g^*)\omega^n,
\]
by the weak convergence of the $\ti{\eta}_k$. 

Let $\e > 0$. Since $(A_k)^i_i((A_k)^i_i)^{*} \to \mathrm{Id}_{r_i}$ and $(A_k)^i_j \to 0$ pointwise, by Egorov's theorem, there exists a set $S_1 \subset S$ on which these convergences are uniform and the measure of $S \setminus S_1$ is as small as desired. By enlarging $S_1$ further, we may also assume that $\| g \|_{L^q(S \setminus S_1)} \leq \epsilon$, where $q := \frac{p}{p-1}$.

Since $\ti{\eta}_k = A_k \eta_k A_k^{*}$, we have
\[
  (\ti{\eta}_k)^i_j = \sum_{a,b} (A_k)^i_a (\eta_k)^a_b ((A_k)_b^j)^{*}, 
\]
we break up the integral $\int_S \mathrm{tr}((\ti{\eta}_k)^i_j g^{*}) \omega^n$ in to a sum of terms of two different types -- the first type has either $i \neq a $ or $j \neq b$, and the second has $i = a$ and $j = b$.

Consider a term of the first type, and assume without loss of generality that $i \neq a$. Then $(A_k)^i_{a} \to 0$ uniformly on $S_1$, so that $|(A_k)^i_{a}| < \frac{\epsilon}{\| g \|_{L^q(S)}}$ on $S_1$ for all $k \gg 0$. Thus, we have
\begin{align*}
  &\left|\int_S\mathrm{tr}((A_k)^i_a (\eta_k)^a_b ((A_k)_b^j)^{*} g^{*})\omega^n\right| \\
  &= \left|\int_{S_1}\mathrm{tr}((A_k)^i_a (\eta_k)^a_b ((A_k)_b^j)^{*} g^{*})\omega^n\right| 
  + \left|\int_{S \setminus S_1}\mathrm{tr}((A_k)^i_a (\eta_k)^a_b ((A_k)_b^j)^{*} g^{*})\omega^n\right| \\
  &\leq \frac{\epsilon}{\| g \|_{L^q(S)}} \int_{S_1} |\eta_k| |g| \omega^n + \int_{S \setminus S_1} |\eta_k| |g| \omega^n \\
  &\leq \frac{\epsilon}{\| g \|_{L^q(S)}} \| \dbar w_k \|_{L^p} \| g \|_{L^q(S)} + \| \dbar w_k \|_{L^p} \| g \|_{L^q(S \setminus S_1)} \\
  &\leq \epsilon \| \dbar w_k \|_{L_p}  + \epsilon \| \dbar w_k \|_{L^p} \\
   &\leq C \epsilon,
 \end{align*}
where the last line follows from the fact that $\dbar w_k \wk \dbar w$ in $L^p$ and thus $\limsup_k \| \dbar w_k \|_{L^p} < \infty $, independent of $k$.

To analyze terms of the second type, consider
\begin{align*}
  &\left|\int_S\mathrm{tr}((A_k)^i_i (\eta_k)^i_j ((A_k)_j^j)^{*} g^{*})\omega^n\right|\\
  &\leq \| g\|_{L^2(S_1)}\left( \int_{S_1} \mathrm{tr}((A_k)^i_i (\eta_n)^i_j ((A_k)_j^j)^{*}(A_k)^j_j ((\eta_k)^i_j)^{*} ((A_k)_i^i)^{*} ) \right)^{\frac{1}{2}} + \int_{S \setminus S_1} | \eta_k | |g| \omega^n.
\end{align*}
The second term is again under control by H\"older's inequality:
\[
\int_{S\setminus S_1} |\eta_k||g|\omega^n \leq \|\pbar w_k\|_{L^p}\|g\|_{L^q(S\setminus S_1)}  \leq C\e.
\]
Since $\norm{g}_{L^2(S_1)} \leq 1$, the first term can then be estimated as:
\begin{align*}
&\int_{S_1} \mathrm{tr}((\eta_k)^i_j ((A_k)_j^j)^{*}(A_k)^j_j ((\eta_k)^i_j)^{*} ((A_k)_i^i)^{*} (A_k)^i_i )\\
&\leq(1+\epsilon)^2\int_{S_1} |(\eta_k)^i_j|^2 \leq (1 + \epsilon)^2\| (\eta_k)^i_j \|_{L^2(S)}^2.
\end{align*}
since $|(A_k)^i_i((A_k)^i_i)^{*} - \mathrm{Id}_{r_i}| < \epsilon$ on $S_1$ for all $i$ and for all $k \gg 0$.

We now have control over both types of terms, so taking $k\rightarrow\infty$ gives:
\[
\left|\int_S \mathrm{tr}(\eta^i_j g^{*}) \omega^n\right| \leq (1 + \epsilon) \left( \liminf_k \| (\eta_k)^i_j \|_{L^2(S)} \right) + C\epsilon.
\]
Using $\| \eta^i_j \|_{L^2(S)} = \sup_{g \in L^2(S) \cap L^{\infty}(S), \| g \|_{L^2(S) = 1}} |\int_S \mathrm{tr}(\eta^i_j g^{*})\omega^n|$, we get that
\[
\| \eta^i_j \|_{L^2(S)} \leq (1+\epsilon) \left( \liminf_k \| (\eta_k)^i_j \|_{L^2(S)} \right) + C\epsilon. 
\]
Taking $\epsilon \to 0$, we get \eqref{eqnSubGoal}, as desired.

\end{proof}

\begin{lemma}
\label{lemLInftyBais}
Let $E|_{U}$ be a trivializable vector bundle. Consider an $L^\infty$ unitary basis $\{ e_i \}$ of $E|_{U}$ i.e.~a unitary basis of $E$ obtained by applying an $L^\infty$ unitary matrix to a smooth unitary basis of $E$. Then, $T_k \wk T$ in $L^p$ for $T_k, T \in \mathrm{End}(E|_U)$ if and only if the same holds for their matrix representatives with respect to the basis $\{ e_i \}$.
\end{lemma}
\begin{proof}
  Note that $T_k \wk T$ if and only if the same holds true for their matrix representatives with respect to a smooth unitary frame of $E$. By abuse of notation, denote these matrices as $T_k, T$ and let $A$ denote the change of basis from this smooth unitary frame to $\{ e_i \}$. Then, we need to show that $T_k \wk T$ if and only if $AT_k A^* \wk ATA^*$. But this is true as $A \in L^\infty$.
\end{proof}

\begin{lemma}
  \label{lemPointwiseConvergenceOfDN}
  Let $r = \sum_{i=1}^m r_i$ and write $B_i^j$ for the (ordered) $r_i\times r_j$-blocks of any $B\in M_{r\times r}(\C)$, as usual. Assume that $D_k,D \in M_{r \times r}(\C)$ and $A_k \in U(r)$ are such that $D_k, D$ are diagonal, with their elements arranged in monotonically decreasing order, and $D = \mathrm{diag}(\lambda_1\mathrm{Id}_{r_1},\dots,\lambda_m\mathrm{Id}_{r_m})$ for some $\lambda_1 > \dots > \lambda_m$. Assume further that $A_kD_kA_k^{*} \to D$ as $k \to \infty$. Then,
  \begin{itemize}
\item $D_k \to D$,
\item $(A_k)_i^i(A_k^*)_i^i \to \mathrm{Id}_{r_i}$ and $(A_k^*)^i_i(A_k)_i^i \to \mathrm{Id}_{r_i}$, and
\item $(A_k)^j_i \to 0$ for $i \neq j$,
\end{itemize}
as $k \to \infty$. 
\end{lemma}
\begin{proof}
  The first part follows from the fact that eigenvalues are a continuous function of the entries of a matrix and that $A_k^{*}D_kA_k$ and $D_k$ have the same eigenvalues.

Since $(D_k - D) \to 0$ and the entries of $A_k$ and $A_k^{*}$ are uniformly bounded, we also get that $A_k(D_k-D)A_k^{*} \to 0$. Thus, we see that $A_kDA_k^{*} \to D$, i.e.
  \[
    \sum_\ell \lambda_i (A_k)_i^\ell(A_k^*)_j^\ell \to 
\begin{cases}
  0 &\text{ if } i \neq j \\
  \lambda_i \mathrm{Id}_{r_{i}} &\text{ if } i = j
\end{cases}.
\]
In particular, picking $i = j = 1$, we get that
\begin{align*}
  \sum_\ell \lambda_{\ell} (A_k)_1^\ell (A_k^*)_1^\ell &\to \lambda_1 \mathrm{Id}_{r_1} \\ 
  \sum_\ell (\lambda_1-\lambda_\ell)(A_k)_1^\ell(A_k^*)_1^\ell &\to 0.  
\end{align*}
Since $A_k$ is unitary, we have that $\sum_\ell (A_k)_i^\ell(A_k^*)_j^\ell = 0$ if $i \neq j$ and $\sum_\ell (A_k)_i^\ell(A_k^*)_i^\ell = \mathrm{Id}_{r_i}$. Since $(\lambda_1 - \lambda_\ell) > 0$ for $\ell \geq 2$ and the diagonal entries of $(A_k)_1^\ell(A_k^*)_1^\ell$ are non-negative, we get that the diagonal entries of $(A_k)_1^\ell(A_k^*)_1^\ell$ converge to 0 for all $\ell \geq 2$ as $k \to \infty$, which implies that $(A_k)_1^\ell \to 0$ for $\ell \geq 2$. Now using $\sum_\ell (A_k)_i^\ell(A_k^*)_i^\ell = \mathrm{Id}_{r_i}$, we also get that $(A_k)_1^1(A_k^*)_1^1 \to \mathrm{Id}_{r_1}$. Applying induction to
\[
  A = 
\begin{pmatrix}
  A_{2}^1 & \dots & A_{2}^m \\
  \vdots & \ddots & \vdots \\
  A_{m}^1 & \dots & A_{m}^m
\end{pmatrix}
\]
and $D = \mathrm{diag}(\lambda_2\mathrm{Id}_{r_2},\dots,\lambda_m\mathrm{Id}_{r_m})$, we get the required result. 
\end{proof}

\begin{lemma}
\label{lemFatou}
Let $A$ be a finite set.  For $k \in \mathbb{N}$ and for  $\alpha \in A$, let $\mu_{k,\alpha}, \mu$ be positive Radon measures on a measure set $Y$ such that
  \[  \liminf_{k\rightarrow\infty} \left( \sum_{\alpha \in A}\mu_{k,\alpha}(S) \right) \geq \mu(S) \]
for all measurable $S \subset Y$. Let $f_{k,\alpha} , f$ be non-negative measurable functions on $U$ such that $f_{k,\alpha} \to f$ pointwise a.e. on $Y$ as $k \to \infty$ for all $\alpha \in A$. Then,
\[
\liminf_{k\rightarrow\infty} \left( \int_{Y} \sum_\alpha f_{k,\alpha} \mu_{k,\alpha} \right) \geq \int_Y f \mu
\]
\end{lemma}
\begin{proof}
  Let $Y_i$ be measurable sets and $\phi = \sum_{i=0}^m a_i I_{Y_i}$ for $0< a_0 < \dots < a_m < \infty$ be a simple function such that $\phi \leq f$ and $\sum_i \mu(Y_i) < \infty$. It is enough to prove that for all such $\phi$, we have $\liminf_k \int_X \sum_{\alpha} f_{k,\alpha} \mu_{k,\alpha} \geq \int_X \phi \mu$.

 Fix $\epsilon > 0$. Let $S_k = \{x \in X \mid f_{\ell,\alpha}(x) \geq (1-\epsilon)\phi \text{ for all } \ell \geq k \text{ and for all } \alpha \}$. Then, $S_k$ is an increasing sequence of sets whose union is $Y$.

 Since $Y_i \subset \bigcup_k S_k$, $\mu(Y_i \cap S_k) \to \mu(Y_i)$. Thus, there exits an $k_0$ such that $\mu(Y_i \cap S_{k_0}) \geq \mu(Y_i) - \frac{1}{2}\epsilon$ for all $i = 0,\dots,m$. By hypothesis, there exists a $\ell_0$ such that
 $$\sum_{\alpha}\mu_{\ell,\alpha}(Y_i \cap S_{k_0}) \geq \mu(Y_i \cap S_{k_0}) - \frac{1}{2}\epsilon \geq \mu(Y_i) - \epsilon$$ for all $\ell \geq \ell_0$ and for all $i$. Since $\sum_{\alpha}\mu_{\ell,\alpha}(Y_i \cap S_\ell) \geq \sum_{\alpha}\mu_{\ell,\alpha}(Y_i \cap S_{k_0}) $ for $\ell \geq k_0$, we get that $\sum_{\alpha}\mu_{\ell,\alpha}(Y_i \cap S_\ell) \geq \mu(Y_i) - \epsilon$ for $\ell \gg 0$.

  Thus, for $k\gg 0$, we get, 
\begin{align*}
  \int_Y \sum_{\alpha} f_{k,\alpha} \mu_{k,\alpha} &\geq \int_{S_k} \sum_{\alpha} f_{k,\alpha} \mu_{k,\alpha}\\
    &\geq (1-\epsilon) \int_{S_k} \phi \sum_{\alpha}\mu_{k,\alpha} \\
     &= (1-\epsilon) \sum_{i=0}^m a_i \left( \sum_{\alpha} \mu_{k,\alpha}(Y_i \cap S_k) \right) \\
    &\geq (1-\epsilon)\sum_{i=0}^m a_i (\mu(Y_i) - \epsilon) \\
  &= (1-\epsilon)\left( \int_S \phi \mu \right) - \epsilon(1-\epsilon)\sum_ia_i.
\end{align*}
Letting $\epsilon \to 0$, we get the required result. 
\end{proof}


\section{Proofs of Theorems~B and~C}\label{sec:final}


We are now ready to prove the remaining results in the Introduction.
\subsection{Proof of Theorem~B}
We must show that if $\mathcal{M}$ fails to be proper on $\mathcal{H}^{1,p}$, then there exists  some $w\in\mathcal{H}^{1,p}$ such that
$\mathcal{M}(tw)\le0$ for all $t\ge0$.

  By assumption, there exists a sequence $w_k \in \mathcal{H}^{1,p}$ such that
  \[
    \mathcal{M}(w_k) < k^{-1}\| w_k \|_{W^{1,p}} - k. 
  \]
  We claim that $\| w_k \|_{W^{1,p}} \to \infty$ as $k\rightarrow\infty$. Recall that
  \[
    \mathcal{M}(w_k) \geq -C \norm{w_k}_{L^1} \geq -C\| w_k \|_{W^{1,p}},
  \]
  for some constant $C> 0$ depending only on $h_0, \omega$. Combining the two bounds for $\mathcal{M}(w_k)$, we get
  \[
- C \| w_k \|_{W^{1,p}} \leq  k^{-1}\| w_k \|_{W^{1,p}} - k, 
\]
and simplifying, we see that
\[
  \| w_k \|_{W^{1,p}} \geq \frac{k}{C+k^{-1}}.
\]
In particular, we get that $\| w_k \|_{W^{1,p}} \to \infty$.

Let $\widetilde{w_k} = \frac{w_k}{\| w_k \|_{W^{1,p}}}$. The $\ti{w_k}$ are clearly bounded in $W^{1,p}$, and hence we can extract a weakly convergent subsequence, which we still denote $\ti{w}_k$. Let $w\in\mathcal{H}^{1,p}$ be the weak limit.

  We claim that $w\not= 0$. To see this, use the reverse Sobolev inequality, Proposition \ref{propReverseSobolev}, to see that:
  \[
  \norm{w_k}_{W^{1,p}} \leq C(\norm{w_k}_{L^{p^*}} + 1)
  \]
  Taking $k$ sufficiently large and then applying the usual Sobolev inequality gives:
  \[
   \frac{1}{2C} \norm{w_k}_{W^{1,p}} \leq \norm{w_k}_{L^{p^*}} \leq C \norm{w_k}_{W^{1,p}},
  \]
  and so we conclude $\frac{1}{2C} \leq \norm{\ti{w}}_{L^{p^*}} \leq C$ for all $k\gg 0$.
  
  By the Rellich--Kondrachov theorem, the embedding of $W^{1,p} \to L^{p^{*}}$ is compact (since $p^* < p'$) and we further have that $\ti{w_k} \to w$ in $L^{p^{*}}$. Thus $\| w \|_{L^{p^{*}}} = \lim_{k \to \infty} \| \ti{w}_k \|_{L^{p^{*}}} >  0$ so that $w \neq 0$. 

  We now examine the (non-trivial) geodesic ray $\{tw\}_{t\geq 0}$; fix some $t > 0$. By convexity of $\mathcal{M}$, we have that
  \begin{align*}
    \mathcal{M}(t\widetilde{w_k}) &\leq \frac{t\mathcal{M}(w_k)}{\| w_k \|_{W^{1,p}}} \leq  \frac{t}{k} - \frac{tk}{\| w_k \|_{W^{1,p}}} \leq \frac{t}{k}, 
  \end{align*}
and hence $\liminf_k \mathcal{M}(t \ti{w_k}) \leq 0$. Now by the semicontinuity of $\mathcal{M}$ (Proposition \ref{prop:Mlsc}), we get that $\mathcal{M}(tw) \leq 0$ for all $t \geq 0$.

\subsection{Proof of Theorem~C}
Now suppose we are given $w \in \mathcal{H}^{1,p}$, such that  $\mathcal{M}(tw) \leq 0$ for all $t \geq 0$. We must prove that $w$ defines a filtration of $E$ by holomorphic subsheaves, and, after resolving the filtration, $e^{tw}h_0$ pulls back to a geodesic ray of the type constructed in Theorem \ref{sheaf_slopes}. Moreover, one of these subsheaves is such that $\mu_{\mathcal{E}_i} \geq \mu_E$. 

  Recall that $\mathcal{M}(tw) = \int_X f_{tw} \omega^n + t \cdot \int_X \mathrm{tr}(\Theta_0 w)\wedge\omega^{n-1}$. Since the first term is positive, and the second term grows linearly in $t$, we have that
  \[
\int_X f_{tw} \omega^n \leq C_1t 
\]
for all $t$ and for some positive constant $C_1$. Consider a unitary basis $\{ e  \}$ with respect to which $w$ is a diagonal matrix. Let the matrix of $w$ and $\dbar w$ with respect to $\{ e \}$ be $\mathrm{diag}(\zeta_1, \zeta_2, \dots, \zeta_r)$ and $\{ \eta_a^b \}_{1 \leq a,b \leq r}$ respectively where we order the eigenvalues as $\zeta_1 \geq \zeta_2 \geq \dots \geq \zeta_r$.
Recall that
\[
f_{tw} = \sum_{1\leq a,b \leq r} | \eta_{b}^a |^2 \frac{e^{t(\zeta_i - \zeta_j)} - 1 - t(\zeta_i - \zeta_j)}{(\zeta_i-\zeta_j)^2}
\]

Since $\frac{e^x - 1 - x}{x^2} \geq \frac{1}{2}$ when $x \geq 0$, we get that
\[
f_{tw} \geq  \frac{t^2}{2} \sum_{a \leq b}  | \eta^a_b |^2. 
\]
Since $\int_X f_{tw} \omega^n < Ct$ for all $t$, we get that $\eta^a_b = 0$ almost everywhere for all $a \leq b$. 

Now we claim that the eigenvalues of $w$ are constant on $X$. To see this, consider a change of basis $A$ from a holomorphic frame of $E$ to $\{e\}$. Then,
\begin{align*}
  \eta &= A^{-1}\dbar(A \cdot D \cdot A^{-1})A \\
  \eta &= A^{-1}(\dbar A) D + \dbar D - D A^{-1} \dbar A,
\end{align*}
where $D$ denotes the diagonal matrix $\mathrm{diag}(\zeta_1,\dots,\zeta_r)$. 
Computing the $(a,b)$-th entry, we get that
  \[
    \eta^a_b =
    \begin{cases}
      (A^{-1}\dbar A)^a_b(\zeta_b - \zeta_a) &\text{ if } a \neq b \\
      \dbar \zeta_a &\text{ if }  a = b
    \end{cases}.
  \]
  Using $\dbar \zeta_a = \eta^a_a  = 0$, we get that the eigenvalues of $w$ are constant on $X$. For the remainder of the proof, let us relabel the eigenvalues of $w$ as $\lambda_1,\dots,\lambda_m$ such that the matrix of $w$ in the basis $\{ e \}$ is given by $\mathrm{diag}(\lambda_1\mathrm{Id}_{r_1},\dots,\lambda_m\mathrm{Id}_{r_m})$,  and $\lambda_1 > \dots > \lambda_m$. We also change our notation so that $\eta_j^i$ and $(A^{-1}\dbar A)^i_j$ now denote the $r_i \times r_j$ block matrices of $\eta$ and $A^{-1}\dbar A$ for $1\leq i,j \leq m$. 

  Now using $\eta^i_j = 0$ for $i < j$, we get that $(A^{-1}\dbar A)^i_j = 0$ for $i < j$. We now use this and Theorem \ref{UY} to show that projection to the sum of eigenspaces of $\lambda_s,\dots,\lambda_m$ gives rise to a holomorphic filtration of $E$. Let $\pi_s$ denote the orthogonal projection to the sum of eigenspaces of $\lambda_s,\dots,\lambda_m$ i.e.~with respect to the basis $\{e\}$, the matrix of $\pi_s$ is
  \[
\Pi_s = \mathrm{diag}(0\cdot \mathrm{Id}_{r_1+\dots+r_{s-1}}, \mathrm{Id}_{r_s+\dots+r_m}).
\]
   Let us also denote $M := A^{-1}\dbar A$ and write
   \[
     M =
     \begin{pmatrix}
       M_1 & 0 \\
       M_2 & M_3
     \end{pmatrix}      
   \]
   where $M_1$ is an $(r_1 + \dots + r_{s-1}) \times (r_1 + \dots + r_{s-1})$ block matrix and $M_3$ is an $(r_s + \dots + r_m) \times (r_s + \dots + r_m)$ block matrix.

  Firstly, we claim that $\pi_s \in \mathcal{H}^{1,2}$. It is clear that $\pi \in L^2$ as the operator norm of $\pi_s$ is $1$. To see that $\dbar \pi_s \in L^2$, the matrix of $\dbar \pi_s$ with respect to the frame $\{ e \}$ is given by
  
   \begin{align*}
     A^{-1}\dbar(A \Pi_s A^{-1}) A &= (A^{-1}\dbar A) \Pi_s - \Pi_s (A^{-1}\dbar A) \\ &=     \begin{pmatrix}
      0 & 0 \\
      M_2 & 0 \\
    \end{pmatrix}
   \end{align*}

   Since $f_w \in L^1(X)$ and since the eigenvalues of $w$ are constant, we get that $\eta_j^i$ and thus $(A^{-1}\dbar A)^i_j$ are in $L^2$. Since the entries of $\dbar \pi_s$ are either $(M_2)^i_j$ or $0$, $\dbar \pi_s \in L^2$.

   We also have that $(\mathrm{I} - \pi_s)\dbar \pi_s$ in the frame $\{ e \}$ is given by
   \[
\begin{pmatrix}
  \mathrm{Id}_{r_1+\dots+r_{s-1}} & 0 \\
  0 & 0
\end{pmatrix}
\cdot
\begin{pmatrix}
  0 & 0 \\
  M_2 & 0
\end{pmatrix} = 0
   \]

   Using Theorem \ref{UY}, we get that the image of $\pi_s$ is a subsheaf $\mathcal{E}_s \subset E$ of rank $r_s+\dots+r_m$, and we get a filtration $E = \mathcal{E}_1 \supset \dots \supset \mathcal{E}_m \supset 0$. Note that $w$ acts by $\lambda_i$ on $\mathcal{F}_i = \mathcal{E}_i/\mathcal{E}_{i+1}$ and thus is of the form constructed in Theorem \ref{sheaf_slopes}. Since $\mathcal{M}(tw) \leq 0$ for all $t \geq 0$, it follows that
   \[
     2\pi\sum_{i=1}^m \lambda_i \mathrm{rk}(\mathcal{F}_i) (\mu_{\mathcal{F}_i} - \mu_E) \leq 0.
\]
   The inequality $\mu_{\mathcal{E}_i} \geq \mu_{\mathcal{E}}$ now follows from Corollary \ref{corNegSlope}.

\printbibliography

\end{document}